\documentclass[a4paper, reqno]{amsart}

\usepackage[utf8]{inputenc}
\usepackage{mathtools}
\usepackage{graphicx}
\usepackage{array}

\newtheorem{lemma}{Lemma}[section]
\newtheorem{theorem}[lemma]{Theorem}

\newcommand{\R}{\mathbb{R}}

\newcommand{\orto}{\mathrm{O}}            
\newcommand{\mparam}{\alpha}              
\newcommand{\mchi}{\chi_{\rm m}}          
\DeclareMathOperator{\vol}{vol}           

\newcommand{\unijac}[2]{P_{#2}^{#1}}      

\newcommand{\optprob}[1]{{\arraycolsep=0pt%
  \begin{array}{r@{\ }l@{\quad}l}
    #1
  \end{array}}}

\newcommand{\defi}[1]{\textit{#1}}

\makeatletter
\newcommand*{\centerfloat}{%
  \parindent \z@
  \leftskip \z@ \@plus 1fil \@minus \textwidth
  \rightskip\leftskip
  \parfillskip \z@skip}
\makeatother

\title{A recursive Lovász theta number for simplex-avoiding sets}

\author{Davi Castro-Silva}
\address{D. Castro-Silva, Department Mathematik/Informatik, Abteilung
  Mathematik, Universität zu Köln, Weyertal 86--90, 50931 Köln,
  Germany.}
\email{dcastros@uni-koeln.de}

\author{Fernando Mário de Oliveira Filho}
\address{F.M. de Oliveira Filho, Delft Institute of Applied Mathematics,
  Delft University of Technology, Mekelweg~4, 2628~CD Delft,
  The Netherlands.}
\email{F.M.deOliveiraFilho@tudelft.nl}

\author{Lucas Slot}
\address{L. Slot, Centrum Wiskunde \& Informatica, Postbus~94079, 1090~GB
  Amsterdam, The Netherlands.}
\email{lucas.slot@cwi.nl}

\author{Frank Vallentin}
\address{F. Vallentin, Department Mathematik/Informatik, Abteilung
  Mathematik, Universität zu Köln, Weyertal 86--90, 50931 Köln,
  Germany.}
\email{frank.vallentin@uni-koeln.de}

\thanks{This project has received funding from the European Union's Horizon
  2020 research and innovation programme under the Marie
  Sk\l{}odowska-Curie agreement No 764759. The fourth author is partially
  supported by the SFB/TRR 191 ``Symplectic Structures in Geometry, Algebra
  and Dynamics'' and by the project ``Spectral bounds in extremal discrete
  geometry'' (project number 414898050), both funded by the DFG}

\subjclass[2010]{05D10, 33C45, 52C10, 90C22}

\date{November 8, 2021}  

\begin{document}

\begin{abstract}
  We recursively extend the Lovász theta number to geometric hypergraphs on
  the unit sphere and on Euclidean space, obtaining an upper bound for the
  independence ratio of these hypergraphs.  As an application we reprove a
  result in Euclidean Ramsey theory in the measurable setting, namely that
  every $k$-simplex is exponentially Ramsey, and we improve existing bounds
  for the base of the exponential.
\end{abstract}

\maketitle
\markboth{D. Castro-Silva, F.M. de Oliveira Filho, L. Slot, and
  F. Vallentin}{A recursive Lovász theta number for simplex-avoiding sets}


\section{Introduction}

The Lovász theta number $\vartheta(G)$ of a finite graph~$G$ satisfies
$\alpha(G) \leq \vartheta(G) \leq \chi(\overline{G})$, where~$\alpha(G)$ is
the independence number of~$G$ and~$\chi(\overline{G})$ is the chromatic
number of the complement~$\overline{G}$ of~$G$, the graph whose edges are
the non-edges of~$G$; the theta number can be computed efficiently using
semidefinite programming.

Originally, Lovász~\cite{Lovasz1979} introduced~$\vartheta$ to determine
the Shannon capacity of the $5$-cycle.  The theta number turned out to be a
versatile tool in optimization, with applications in combinatorics and
geometry.  It is related to spectral bounds like Hoffman's bound, as noted
by Lovász in his paper (cf.~Bachoc, DeCorte, Oliveira, and
Vallentin~\cite{BachocDOV2014}), and also to Delsarte's linear programming
bound in coding theory, as observed independently by McEliece, Rodemich,
and Rumsey~\cite{McElieceRR1978} and Schrijver~\cite{Schrijver1979}.

Bachoc, Nebe, Oliveira, and Vallentin~\cite{BachocNOV2009}
extended~$\vartheta$ to infinite geometric graphs on compact metric spaces.
They also showed that this extension leads to the classical linear
programming bound for spherical codes of Delsarte, Goethals, and
Seidel~\cite{DelsarteGS1977}; the linear programming bound of Cohn and
Elkies for the sphere-packing density~\cite{CohnE2003} can also be seen as
an appropriate extension of~$\vartheta$~\cite{LaatOV2014, OliveiraV2019}.
These many applications illustrate the power of the Lovász theta number as
a unifying concept in optimization; Goemans~\cite{Goemans1997} even
remarked that ``it seems all paths lead to~$\vartheta$!''.

We will show how a recursive variant of~$\vartheta$ can be used to find
upper bounds for the independence ratio of geometric hypergraphs on the
sphere and on Euclidean space; this will lead to new bounds for a problem
in Euclidean Ramsey theory.


\subsection{Unit sphere}

We call a set~$\{x_1, \ldots, x_k\}$ of~$k \geq 2$ points on the
$(n-1)$-dimensional unit sphere~$S^{n-1} = \{\, x \in \R^n : \|x\| = 1\,\}$
a \defi{$(k, t)$-simplex} if~$x_i \cdot x_j = t$ for all~$i \neq j$.  The
convex hull of a $(k, t)$-simplex has dimension~$k-1$.  There is a
$(k, t)$-simplex in~$S^{n-1}$ for every~$k \leq n$
and~$t \in [-1/(k-1), 1)$.

Fix~$n \geq k \geq 2$ and~$t \in [-1/(k-1), 1)$.  A set of points
in~$S^{n-1}$ \defi{avoids} $(k, t)$-simplices if no~$k$ points in the set
form a $(k, t)$-simplex.  We are interested in the parameter
\[
  \mparam(S^{n-1}, k, t) = \sup\{\, \omega(I) : \text{$I \subseteq S^{n-1}$
    is measurable and avoids $(k, t)$-simplices}\,\},
\]
where~$\omega$ is the surface measure on the sphere normalized so the total
measure is~$1$.  This is the independence ratio of the hypergraph whose
vertex set is~$S^{n-1}$ and whose edges are all $(k, t)$-simplices.

In~\S\ref{sec:sphere} we will define the
parameter~$\vartheta(S^{n-1}, k, t)$ recursively as the optimal value of
the problem
\[
  \optprob{\sup&\int_{S^{n-1}} \int_{S^{n-1}} f(x\cdot y)\,
    d\omega(y) d\omega(x)\\[1ex]
    &f(1) = 1,\\
    &f(t) \leq \vartheta(S^{n-2}, k-1, t/(1+t)),\\
    &\text{$f \in C([-1, 1])$ is a function of positive type
      for~$S^{n-1}$} }
\]
for~$k \geq 3$.  The base of the recursion is~$k = 2$:
$\vartheta(S^{n-1}, 2, t)$ is the optimal value of the problem above
when~``$f(t) \leq \vartheta(S^{n-2}, k- 1, t/(1+t))$'' is replaced
by~``$f(t) = 0$''.

From Theorem~\ref{thm:sphere-nplus} below it follows
that~$\vartheta(S^{n-1},k,t) \geq \mparam(S^{n-1}, k, t)$.  Using extremal
properties of ultraspherical polynomials, an explicit formula can be
computed for this bound, as shown in Theorem~\ref{thm:theta-k}.


\subsection{Euclidean space}

Transferring these concepts from the compact unit sphere to the non-compact
Euclidean space requires a bit of care; this is done
in~\S\ref{sec:Euclidean}.

A \defi{unit $k$-simplex} in~$\R^n$ is a set~$\{x_1, \ldots, x_k\}$
of~$k \leq n + 1$ points such that $\|x_i-x_j\|=1$ for all~$i \neq j$.  As
before, the dimension of the convex hull of a unit $k$-simplex is~$k-1$.  A
set of points in~$\R^n$ \defi{avoids} unit $k$-simplices if no~$k$ points
in the set form a unit $k$-simplex.  We are interested in the parameter
\[
  \mparam(\R^n, k) = \sup\{\, \overline{\delta}(I) : \text{$I \subseteq
    \R^n$ is measurable and avoids unit $k$-simplices}\,\},
\]
where~$\overline{\delta}(X)$ is the \defi{upper density}
of~$X \subseteq \R^n$, that is,
\[
  \overline{\delta}(X) = \limsup_{T \to \infty} \frac{\vol(X \cap [-T,
    T]^n)}{\vol [-T, T]^n}.
\]

Again, this parameter has an interpretation in terms of the independence
ratio of a hypergraph on the Euclidean space and again we can bound the
independence ratio from above by an appropriately defined
parameter~$\vartheta(\R^n,k)$.  Theorem~\ref{thm:theta-rn-k} below gives an
explicit expression for~$\vartheta(\R^n, k)$ in terms of Bessel functions
and ultraspherical polynomials.


\subsection{Euclidean Ramsey theory}
\label{sec:Ramsey}

The central question of Euclidean Ramsey theory is: given a finite
configuration~$P$ of points in~$\R^n$ and an integer~$r \geq 1$, does every
$r$-coloring of~$\R^n$ contain a monochromatic congruent copy of~$P$?

The simplest point configurations are unit $k$-simplices, which are known
to have the exponential Ramsey property: the minimum number~$\chi(\R^n, k)$
of colors needed to color the points of~$\R^n$ in such a way that there are
no monochromatic unit $k$-simplices grows exponentially in~$n$.  This was
first proved by Frankl and Wilson~\cite{FranklW1981} for~$k = 2$ and by
Frankl and Rödl~\cite{FranklR1987} for~$k > 2$.  Results in this area are
usually proved by the linear algebra method; see also
Sagdeev~\cite{Sagdeev2018}.

Recently, Naslund~\cite{Naslund2020} used the slice-rank method from the
work of Croot, Lev, and Pach~\cite{CrootLP2017} and Ellenberg and
Gijswijt~\cite{EllenbergG2017} on the cap-set problem\footnote{The
  slice-rank method is only implicit in the original works; the actual
  notion of slice rank for a tensor was introduced by Tao in a blog
  post~\cite{Tao2016}.} to prove that
\[
  \chi(\R^n, 3) \geq (1.01466 + o(1))^n.
\]
This is the best lower bound known at the moment.

For simplices of higher dimension, Sagdeev~\cite{Sagdeev2018b} used a
quantitative version of the Frankl-Rödl theorem to show that
\[
  \chi(\R^n, k) \geq \biggl(1 + \frac{1}{2^{2^{k+3}}} + o(1)\biggr)^n.
\]

Denote by~$H(n, k)$ the \defi{unit-distance hypergraph}, namely the
$k$-uniform hypergraph whose vertex set is~$\R^n$ and whose edges are all
unit $k$-simplices.  The parameter~$\chi(\R^n, k)$ is the chromatic number
of this hypergraph.  A theorem of de Bruijn and Erdős~\cite{BruijnE1951}
shows\footnote{De Bruijn and Erdős consider only~$k = 2$, but their result
  can be easily generalized to~$k \geq 3$.} that computing~$\chi(\R^n, k)$
is a combinatorial problem: the chromatic number of~$H(n, k)$ is the
maximum chromatic number of any finite subgraph of~$H(n, k)$.

Determining~$\chi(\R^n, 2)$ is known as the Nelson-Hadwiger problem.  The
problem was proposed by Nelson in~1950
(cf.~Soifer~\cite[Chapter~3]{Soifer2009}), who used the Moser spindle, a
$4$-chromatic $7$-vertex subgraph of~$H(2, 2)$, to show
that~$\chi(\R^2, 2) \geq 4$.  Isbell (cf.~Soifer, ibid.), also in~1950,
proved that~$\chi(\R^2, 2) \leq 7$ by constructing a coloring of~$H(2, 2)$.

The difficulty of finding subgraphs of~$H(2, 2)$ with chromatic number
higher than~$4$ led Falconer~\cite{Falconer1986} to define the
\defi{measurable chromatic number}~$\mchi(\R^n, 2)$ by requiring the color
classes to be Lebesgue-measurable sets; we define~$\mchi(\R^n, k)$ likewise
for~$k \geq 3$.  Of course, $\mchi(\R^n, k) \geq \chi(\R^n, k)$, but it is
not known whether the two numbers differ.  Falconer could show
that~$\mchi(\R^2, 2) \geq 5$, whereas a proof that~$\chi(\R^2, 2) \geq 5$
was only obtained more than three decades later by de Grey~\cite{Grey2018},
who found by computer a $5$-chromatic subgraph of~$H(2, 2)$ with~1581
vertices.

The restriction to measurable color classes also helps improving asymptotic
lower bounds.  Frankl and Wilson~\cite{FranklW1981} give a combinatorial
proof that
\[
  \chi(\R^n, 2) \geq (1.2 + o(1))^n.
\]
By using analytical techniques, Bachoc, Passuello, and
Thiery~\cite{BachocPT2015} could show that
\[
  \mchi(\R^n, 2) \geq (1.268 + o(1))^n.
\]

Similarly, the analytical tools developed in this paper can also be used to
improve lower bounds for~$\mchi(\R^n, k)$.  Since
\begin{equation}
  \label{eq:chi-alpha}
  \mparam(\R^n, k) \mchi(\R^n, k) \geq 1,
\end{equation}
any upper bound for~$\mparam(\R^n, k)$ gives a lower bound
for~$\mchi(\R^n, k)$, hence
\[
  \mchi(\R^n, k) \geq \lceil 1 / \vartheta(\R^n, k)\rceil.
\]
In~\S\ref{sec:decay} we analyze the upper bounds~$\vartheta(S^{n-1}, k, t)$
for simplex-avoiding sets on the sphere and~$\vartheta(\R^n, k)$ for
simplex-avoiding sets on Euclidean space by using properties of
ultraspherical polynomials, obtaining the following theorem.

\begin{theorem}
  \label{thm:main-decay}
  If~$k \geq 2$, then:
  \begin{itemize}
  \item[(i)] for every~$t \in (0, 1)$, there is a
    constant~$c = c(k, t) \in (0, 1)$ such
    that $\vartheta(S^{n-1}, k, t) \leq (c + o(1))^n$;

  \item[(ii)] there is a constant~$c = c(k) \in (0, 1)$ such
    that~$\vartheta(\R^n, k) \leq (c + o(1))^n$.
  \end{itemize}
\end{theorem}

From this theorem we get an exponential lower bound for~$\mchi(\R^n, k)$.
Rigorous estimates of the constant~$c$ then yield significantly better
lower bounds for~$\mchi(\R^n, k)$ than those coming from~$\chi(\R^n, k)$.

Indeed, in the case $k=3$ we obtain (see \S\ref{sec:exp-bounds})
\[
  \mparam(\R^n, 3) \leq (0.95622 + o(1))^n,
\]
and so
\[
  \mchi(\R^n, 3) \geq (1.04578 + o(1))^n.
\]
We also obtain the rougher estimate
\[
  \mparam(\R^n, k) \leq \biggl(1 - \frac{1}{9(k-1)^2} + o(1)\biggr)^n,
\]
valid for all $k \geq 3$, which immediately implies
\[
  \mchi(\R^n, k) \geq \biggl(1 + \frac{1}{9(k-1)^2} + o(1)\biggr)^n.
\]

Though our lower bounds for~$\mchi(\R^n, k)$ do not necessarily hold
for~$\chi(\R^n, k)$, they do imply some structure for general colorings.
If a coloring of~$H(n, k)$ uses fewer than~$1 / \mparam(\R^n, k)$ colors,
then the closure of one of the color classes is a measurable set with
density greater than~$\mparam(\R^n, k)$, and so it contains a unit
$k$-simplex.  This means that in such a coloring there are monochromatic
$k$-point configurations arbitrarily close to unit $k$-simplices.


\subsection{Notation and preliminaries}

We will denote the Euclidean inner product between~$x$, $y \in \R^n$
by~$x\cdot y$.  The surface measure on the sphere is denoted by~$\omega$
and is always normalized so the total measure is~$1$.

We always normalize the Haar measure on a compact group so the total
measure is~$1$.  By~$\orto(n)$ we denote the group of $n\times n$
orthogonal matrices.  If~$X \subseteq S^{n-1}$ is any measurable set and
if~$\mu$ is the Haar measure on~$\orto(n)$, then for every~$e \in S^{n-1}$
we have
\[
  \mu(\{\, T \in \orto(n) : Te \in X\, \}) = \omega(X).
\]




We will need the following technical lemma, which will be applied to the
sphere and the torus.  For a proof, see Lemma~5.5 in DeCorte, Oliveira, and
Vallentin~\cite{DeCorteOV2020}.

\begin{lemma}
  \label{lem:averaging}
  Let~$V$ be a metric space and~$\Gamma$ be a compact group that acts
  transitively on~$V$; let~$\nu$ be a finite Borel measure on~$V$ that is
  positive on open sets.  Denote by~$\mu$ the Haar measure on~$\Gamma$.  If
  the metric on~$V$ and the measure~$\nu$ are $\Gamma$-invariant and
  if~$f \in L^2(V; \nu)$, then the function~$K\colon V \times V \to \R$
  such that
  \[
    K(x, y) = \int_\Gamma f(\sigma x) f(\sigma y)\, d\mu(\sigma)
  \]
  is continuous.
\end{lemma}


\section{Simplex-avoiding sets on the sphere}
\label{sec:sphere}

We call a continuous kernel~$K\colon S^{n-1} \times S^{n-1} \to \R$
\defi{positive} if for every finite set~$U \subseteq S^{n-1}$ the
matrix~$\bigl(K(x, y)\bigr)_{x,y \in U}$ is positive semidefinite.  A
continuous function $f\colon [-1, 1] \to \R$ is of \defi{positive type
  for~$S^{n-1}$} if the kernel~$K\in C(S^{n-1} \times S^{n-1})$ given by
$K(x, y) = f(x\cdot y)$ is positive.

Fix~$n \geq k \geq 3$ and~$t \in [-1/(k-1), 1)$.  For any~$\gamma \geq 0$,
consider the optimization problem
\begin{equation}
  \label{eq:general-sphere-nplus}
  \optprob{\sup&\int_{S^{n-1}} \int_{S^{n-1}} f(x\cdot y)\,
    d\omega(y) d\omega(x)\\[1ex]
    &f(1) = 1,\\
    &f(t) \leq \gamma,\\
    &\text{$f \in C([-1, 1])$ is a function of positive type for~$S^{n-1}$.}
  }
\end{equation}

\begin{theorem}
  \label{thm:sphere-nplus}
  Fix~$n \geq k \geq 3$, $t \in [-1/(k-1), 1)$.
  If~$\gamma \geq \mparam(S^{n-2}, k - 1, t/(1+t))$, then the optimal value
  of~\eqref{eq:general-sphere-nplus} is an upper bound
  for~$\mparam(S^{n-1}, k, t)$.
\end{theorem}

\begin{proof}
Let~$I \subseteq S^{n-1}$ be a measurable set that avoids
$(k, t)$-simplices and assume~$\omega(I) > 0$.  Consider the
kernel~$K\colon S^{n-1} \times S^{n-1} \to \R$ such that
\[
  K(x, y) = \int_{\orto(n)} \chi_I(Tx) \chi_I(Ty)\, d\mu(T),
\]
where~$\chi_I$ is the characteristic function of~$I$ and where~$\mu$ is the
Haar measure on~$\orto(n)$.

By taking~$V = S^{n-1}$ and~$\Gamma = \orto(n)$ in
Lemma~\ref{lem:averaging}, we see that~$K$ is continuous.  By
construction,~$K$ is also positive and invariant, that
is,~$K(Tx, Ty) = K(x, y)$ for all~$T \in \orto(n)$ and
$x$,~$y \in S^{n-1}$.  Such kernels are of the form
$K(x, y) = g(x\cdot y)$, where $g \in C([-1, 1])$ is of positive type
for~$S^{n-1}$.  Note that
\[
  K(x, x) = \int_{\orto(n)} \chi_I(Tx)\, d\mu(T) = \omega(I),
\]
so~$g(1) = \omega(I) > 0$.

Set~$f = g / g(1)$.  Immediately we have that~$f$ is continuous and of
positive type and that~$f(1) = 1$; moreover
\[
  \int_{S^{n-1}} \int_{S^{n-1}} f(x\cdot y)\, d\omega(y) d\omega(x) =
  \omega(I).
\]
Hence, if we show that~$f(t) \leq \gamma$, the theorem will follow.

If~$x \in S^{n-1}$ is a point in a $(k, t)$-simplex, all other points in
the simplex are in~$U_{x,t} = \{\, y \in S^{n-1} : y \cdot x = t\,\}$.
Note that~$U_{x,t}$ is an $(n-2)$-dimensional sphere with
radius~$(1 - t^2)^{1/2}$; let~$\nu$ be the surface measure on~$U_{x,t}$
normalized so the total measure is~$1$.

If~$T \in \orto(n)$ is any orthogonal matrix, then~$TI$ avoids
$(k, t)$-simplices.  Hence if~$x \in TI$, then~$TI \cap U_{x,t}$ cannot
contain~$k-1$ points with pairwise inner product~$t$, and so
$\nu(TI \cap U_{x,t}) \leq \mparam(S^{n-2}, k - 1, t / (1 + t)) \leq
\gamma$.  Indeed, the natural bijection between~$U_{x, t}$ and~$S^{n-2}$
maps pairs of points with inner product~$t$ to pairs of points with inner
product~$t / (1+t)$, and so~$TI \cap U_{x, t}$ is mapped to a subset
of~$S^{n-2}$ avoiding $(k-1, t/(1+t))$-simplices.



Now fix~$x \in S^{n-1}$ and note that
\[
  \begin{split}
    g(t) = \int_{U_{x,t}} K(x, y)\, d\nu(y) &= \int_{U_{x,t}} \int_{\orto(n)}
    \chi_I(Tx)
    \chi_I(Ty)\, d\mu(T) d\nu(y)\\
    &=\int_{\orto(n)} \chi_I(Tx) \int_{U_{x,t}} \chi_I(Ty)\, d\nu(y) d\mu(T)\\
    &\leq \gamma \omega(I),
  \end{split}
\]
whence~$f(t) \leq \gamma$, and we are done.
\end{proof}

One obvious choice for~$\gamma$ in Problem~\eqref{eq:general-sphere-nplus}
is the bound given by the same problem for $(k-1, t/(1+t))$-simplices.  The
base for the recursion is~$k=2$: then we need an upper bound for the
measure of a set of points on the sphere that avoids pairs of points with a
fixed inner product.  Such a bound was given by Bachoc, Nebe, Oliveira, and
Vallentin~\cite{BachocNOV2009} and looks very similar
to~\eqref{eq:general-sphere-nplus}.  They show that, for~$n \geq 2$
and~$t \in [-1, 1)$, the optimal value of the following optimization
problem is an upper bound for~$\mparam(S^{n-1}, 2, t)$:
\begin{equation} 
  \label{eq:sphere-theta}
  \optprob{\sup&\int_{S^{n-1}} \int_{S^{n-1}} f(x\cdot y)\,
    d\omega(y) d\omega(x)\\[1ex]
    &f(1) = 1,\\
    &f(t) = 0,\\
    &\text{$f \in C([-1, 1])$ is a function of positive type
      for~$S^{n-1}$.}
  }
\end{equation}

Let~$\vartheta(S^{n-1}, 2, t)$ denote the optimal value of the optimization
problem above, so~$\vartheta(S^{n-1}, 2, t) \geq \mparam(S^{n-1}, 2, t)$.
For~$k \geq 3$ and~$t \in [-1/(k-1), 1)$, let~$\vartheta(S^{n-1}, k, t)$ be
the optimal value of Problem~\eqref{eq:general-sphere-nplus}
when~$\gamma = \vartheta(S^{n-2}, k-1, t/(1+t))$.  We then
have~$\vartheta(S^{n-1}, k, t) \geq \mparam(S^{n-1}, k, t)$.

There is actually a simple analytical expression
for~$\vartheta(S^{n-1}, k, t)$, as we see now.  For~$n \geq 2$
and~$j \geq 0$, let~$\unijac{n}{j}$ denote the Jacobi polynomial with
parameters~$\alpha = \beta = (n-3)/2$ and degree~$j$, normalized
so~$\unijac{n}{j}(1) = 1$ (for background on Jacobi polynomials, see the
book by Szegö~\cite{Szego1975}).

In Theorem~6.2 of Bachoc, Nebe, Oliveira,
and Vallentin~\cite{BachocNOV2009} it is shown that for
every~$t \in [-1, 1)$ there is some~$j \geq 0$ such
that~$\unijac{n}{j}(t) < 0$.  Theorem~8.21.8 in the book by
Szegö~\cite{Szego1975} implies that, for every~$t \in (-1, 1)$,
\begin{equation}
  \label{eq:jac-limit}
  \lim_{j\to\infty} \unijac{n}{j}(t) = 0.
\end{equation}
Hence, for every~$t \in (-1, 1)$ we can define
\begin{equation}
  \label{eq:mnt}
  M_n(t) = \min\{\, \unijac{n}{j}(t) : j \geq 0\,\},
\end{equation}
and we see that~$M_n(t) < 0$.  With this we
have~\cite[Theorem~6.2]{BachocNOV2009}
\begin{equation}
  \label{eq:theta-optimal}
  \vartheta(S^{n-1}, 2, t) = \frac{-M_n(t)}{1 - M_n(t)}.
\end{equation}
The expression for~$\vartheta(S^{n-1}, k, t)$ is very similar.

\begin{theorem}
  \label{thm:theta-k}
  If~$n \geq k \geq 3$ and if~$t \in [-1/(k-1), 1)$, then
  \begin{equation}
    \label{eq:sphere-theta-expr}
    \vartheta(S^{n-1}, k, t) = \frac{\vartheta(S^{n-2}, k-1, t/(1+t)) -
      M_n(t)}{1-M_n(t)}.
  \end{equation}
\end{theorem}

The proof requires the following characterization of functions of positive
type due to Schoenberg~\cite{Schoenberg1942}: a
function~$f\colon [-1, 1] \to \R$ is continuous and of positive type
for~$S^{n-1}$ if and only if there are nonnegative numbers~$f_0$, $f_1$,
\dots\ such that~$\sum_{j=0}^\infty f_j < \infty$ and
\begin{equation}
  \label{eq:schoenberg-sphere}
  f(t) = \sum_{j=0}^\infty f_j \unijac{n}{j}(t),
\end{equation}
with uniform convergence in~$[-1, 1]$.

\begin{proof}[Proof of Theorem \ref{thm:theta-k}]
The orthogonality of the Jacobi polynomials~$\unijac{n}{j}$ implies in
particular that, if~$j \geq 1$, then
\[
  \int_{S^{n-1}}\int_{S^{n-1}} \unijac{n}{j}(x\cdot y)\, d\omega(y)
  d\omega(x) = 0.
\]
Use this and Schoenberg's characterization of positive type functions to
rewrite~\eqref{eq:general-sphere-nplus}
with~$\gamma = \vartheta(S^{n-2}, k-1, t/(1+t))$, obtaining the equivalent
problem
\[
  \optprob{\sup&f_0\\
    &\sum_{j=0}^\infty f_j = 1,\\[1ex]
    &\sum_{j=0}^\infty f_j \unijac{n}{j}(t) \leq \vartheta(S^{n-2}, k-1,
    t/(1+t)),\\[1ex]
    &\text{$f_j \geq 0$ for all~$j \geq 0$.}
  }
\]

To solve this problem, note that
\[
  \sum_{j=0}^\infty f_j \unijac{n}{j}(t)
\]
is a convex combination of the numbers~$\unijac{n}{j}(t)$.  We want to keep
this convex combination below~$\vartheta(S^{n-2}, k-1, t/(1+t))$ while
maximizing~$f_0$.  The best way to do so is to concentrate all the weight
of the combination on~$f_0$ and~$f_{j^*}$, where~$j^*$ is such
that~$\unijac{n}{j^*}(t)$ is the most negative number appearing in the
convex combination, that is,~$\unijac{n}{j^*}(t) = M_n(t)$.  Now solve the
problem using only the variables~$f_0$ and~$f_{j^*}$ to get the optimal
value as given in the statement of the theorem.
\end{proof}

The expression for~$\vartheta(S^{n-1}, k, 0)$ is particularly simple.
Indeed, for~$n \geq 2$ it follows from the recurrence relation for the
Jacobi polynomials that~$M_n(0) = \unijac{n}{2}(0) = -1/(n-1)$, whence
\[
  \vartheta(S^{n-1}, k, 0) = (k-1)/n.
\]

Figure~\ref{fig:sphere-bounds} shows the behavior
of~$\vartheta(S^{n-1}, 3, t)$ for a few values of~$n$ as~$t$ changes.
Plots for~$k > 3$ are very similar.

\begin{figure}[tb]
  \centerfloat
  \includegraphics{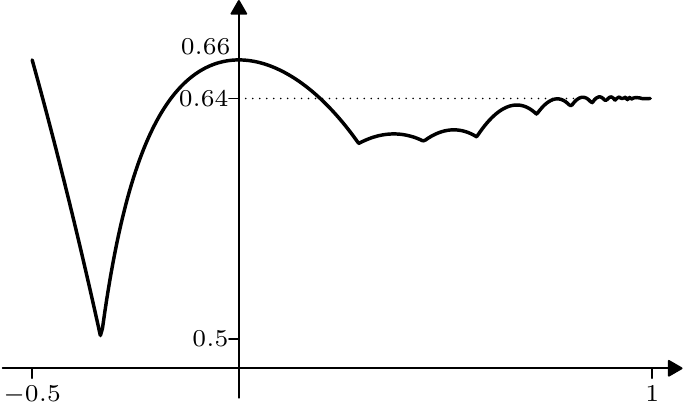}\qquad\includegraphics{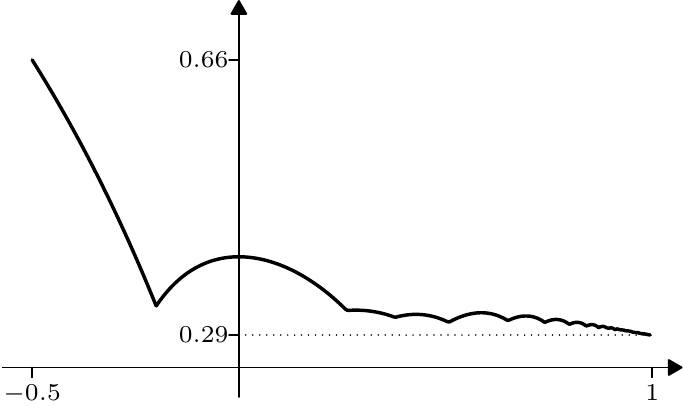}

  \caption{Plots of~$\vartheta(S^{n-1}, 3, t)$ for~$t \in [-0.5,1]$ and
    $n = 3$ (left) and~$5$ (right).}
  \label{fig:sphere-bounds}
\end{figure}


\section{Simplex-avoiding sets in Euclidean space}
\label{sec:Euclidean}

An optimization problem similar to~\eqref{eq:general-sphere-nplus} provides
an upper bound for~$\mparam(\R^n, k)$.  To introduce it, we need some
definitions and facts from harmonic analysis on~$\R^n$; for background, see
e.g.\ the book by Reed and Simon~\cite{ReedS1975}.

A continuous function~$f\colon\R^n \to \R$ is of \defi{positive type} if
for every finite set~$U \subseteq \R^n$ the matrix
$\bigl(f(x-y)\bigr)_{x,y \in U}$ is positive semidefinite.  Such a
function~$f$ has a well-defined \defi{mean value}
\[
  M(f) = \lim_{T \to \infty} \frac{1}{\vol [-T,T]^n} \int_{[-T, T]^n}
  f(x)\, dx.
\]
We say that a function~$f\colon\R^n \to \R$ is \defi{radial} if~$f(x)$
depends only on~$\|x\|$.  In this case, for~$t \geq 0$ we denote by~$f(t)$ the
common value of~$f$ for vectors of norm~$t$.

Fix~$n \geq 2$ and~$k \geq 3$ such that~$k \leq n + 1$.  For every~$\gamma
\geq 0$, consider the optimization problem
\begin{equation}
  \label{eq:general-euclidean-nplus}
  \optprob{\sup&M(f)\\
    &f(0) = 1,\\
    &f(1) \leq \gamma,\\
    &\text{$f\colon \R^n \to \R$ is continuous, radial, and of positive
      type.}
  }
\end{equation}
We have the analogue of Theorem~\ref{thm:sphere-nplus}:

\begin{theorem} 
  \label{thm:Rnanalog}
  Fix~$n \geq 2$ and~$k \geq 3$ such that~$k \leq n + 1$.  If~$\gamma \geq
  \mparam(S^{n-1}, k-1, 1/2)$, then the optimal value
  of~\eqref{eq:general-euclidean-nplus}  is an upper bound
  for~$\mparam(\R^n, k)$.
\end{theorem}

We need a few facts about periodic sets and functions.  A
set~$X \subseteq \R^n$ is \defi{periodic} if it is invariant under some
lattice~$\Lambda$, that is, if~$X + v = X$ for all~$v \in \Lambda$.
Similarly, a function~$f\colon\R^n \to \R$ is \defi{periodic} if there is a
lattice~$\Lambda$ such that~$f(x+v) = f(x)$ for all~$v \in \Lambda$.  We
say that~$\Lambda$ is a \defi{periodicity lattice} of~$X$ or~$f$.  A
periodic function~$f$ with periodicity lattice~$\Lambda$ can be seen as a
function on the torus~$\R^n / \Lambda$; its mean value is
\[
  \frac{1}{\vol(\R^n/\Lambda)} \int_{\R^n/\Lambda} f(x)\, dx.
\]

\begin{proof}[Proof of Theorem \ref{thm:Rnanalog}]
Let~$I \subseteq \R^n$ be a measurable set of positive upper density
avoiding unit $k$-simplices.  The first step is to see that we can assume
that~$I$ is periodic.  Indeed, fix~$R > 1/2$.  Erase a border of
width~$1/2$ around $I \cap [-R, R]^n$ and paste the resulting set
periodically in such a way that there is an empty gap of width~$1$ between
any two pasted copies.  The resulting periodic set still avoids unit
$k$-simplices and is measurable.  Its upper density is
\[
  \frac{\vol(I \cap [-R+1/2, R-1/2]^n)}{\vol [-R,R]^n};
\]
by taking~$R$ large enough, we can make this density as close as we want to
the upper density of~$I$.

Assume then that~$I$ is periodic, so its characteristic function~$\chi_I$
is also periodic; let~$\Lambda$ be a periodicity lattice of~$I$.  Set
\[
  g(x) = \frac{1}{\vol(\R^n / \Lambda)} \int_{\R^n / \Lambda} \chi_I(y)
  \chi_I(x + y)\, dy.
\]
Lemma~\ref{lem:averaging} with~$V = \Gamma = \R^n / \Lambda$ applied
to~$\chi_I$ implies that~$g$ is continuous.  Direct verification yields
that~$g$ is of positive type,~$g(0) = \overline{\delta}(I)$,
and~$M(g) = \overline{\delta}(I)^2$.

Now set
\[
  f(x) = \overline{\delta}(I)^{-1} \int_{\orto(n)} g(Tx)\, d\mu(T),
\]
where~$\mu$ is the Haar measure on~$\orto(n)$.  The function~$f$ is
continuous, radial, and of positive type.  Moreover,~$f(0) = 1$
and~$M(f) = \overline{\delta}(I)$.  If we show that~$f(1) \leq \gamma$,
then~$f$ is a feasible solution of~\eqref{eq:general-euclidean-nplus}
with~$M(f) = \overline{\delta}(I)$, and so the theorem will follow.

To see that~$f(1) \leq \gamma$, note that if~$x$ is a point of a unit
$k$-simplex in~$\R^n$, then all the other points in the simplex lie on the
unit sphere~$x + S^{n-1}$ centered at~$x$.  Hence if~$x \in I$,
then~$I \cap (x + S^{n-1})$ is a measurable subset of~$x + S^{n-1}$ that
avoids $(k - 1, 1/2)$-simplices, and so the measure
of~$I \cap (x + S^{n-1})$ as a subset of the unit sphere is at
most~$\mparam(S^{n-1}, k-1, 1/2)$.  Hence if~$\xi \in \R^n$ is any unit
vector, then
\[
  \begin{split}
    f(1) &= \overline{\delta}(I)^{-1} \int_{\orto(n)} g(T\xi)\, d\mu(T)\\
    &=\overline{\delta}(I)^{-1} \int_{\orto(n)} \frac{1}{\vol(\R^n / \Lambda)}
    \int_{\R^n / \Lambda} \chi_I(x) \chi_I(T\xi + x)\, dx d\mu(T)\\
    &=\overline{\delta}(I)^{-1} \frac{1}{\vol(\R^n / \Lambda)} \int_{\R^n /
      \Lambda} \chi_I(x) \int_{\orto(n)} \chi_I(T\xi + x)\, d\mu(T)
    dx\\
    &\leq \mparam(S^{n-1}, k-1, 1/2) \leq \gamma,
  \end{split}
\]
as we wanted.
\end{proof}

Denote by~$\vartheta(\R^n, k)$ the optimal value
of~\eqref{eq:general-euclidean-nplus} when
setting~$\gamma = \vartheta(S^{n-1}, k-1, 1/2)$.
Then~$\vartheta(\R^n, k) \geq \mparam(\R^n, k)$.

An expression akin to the one for~$\vartheta(S^{n-1}, k, t)$ can be derived
for~$\vartheta(\R^n, k)$.  For~$n \geq 2$, let
\[
  \Omega_n(0) = 1\qquad\text{and}\qquad\Omega_n(u) = \Gamma(n/2)
  (2/u)^{(n-2)/2} J_{(n-2)/2}(u)\quad\text{for~$u > 0$},
\]
where~$J_\alpha$ is the Bessel function of the first kind with
parameter~$\alpha$.  Let~$m_n$ be the global minimum of~$\Omega_n$, which is
a negative number (cf.~Oliveira and Vallentin~\cite{OliveiraV2010}).  The
following theorem is the analogue of Theorem~\ref{thm:theta-k}.

\begin{theorem}
  \label{thm:theta-rn-k}
  For~$n \geq 2$ we have
  \[
    \vartheta(\R^n, k) = \frac{\vartheta(S^{n-1}, k-1, 1/2) - m_n}{1-m_n}.
  \]
\end{theorem}

The proof uses again a theorem of Schoenberg~\cite{Schoenberg1938}, that
this time characterizes radial and continuous functions of positive type
on~$\R^n$: these are the functions~$f\colon\R^n \to \R$ such that
\begin{equation}
  \label{eq:schoenberg-rn}
  f(x) = \int_0^\infty \Omega_n(z\|x\|)\, d\nu(z)
\end{equation}
for some finite Borel measure~$\nu$.

\begin{proof}
If~$f$ is given as in~\eqref{eq:schoenberg-rn}, then~$M(f) = \nu(\{0\})$
(see e.g.~\S6.2 in DeCorte, Oliveira, and Vallentin~\cite{DeCorteOV2020}).
Using Schoenberg's theorem, we can
rewrite~\eqref{eq:general-euclidean-nplus} (with
$\gamma = \vartheta(S^{n-1}, k-1, 1/2)$) equivalently as:
\[
  \optprob{\sup&\nu(\{0\})\\
    &\nu([0, \infty)) = 1,\\[1ex]
    &\int_0^\infty \Omega_n(z)\, d\nu(z) \leq \vartheta(S^{n-1}, k-1,
    1/2),\\[1ex]
    &\text{$\nu$ is a Borel measure.}
  }
\]

We are now in the same situation as in the proof of
Theorem~\ref{thm:theta-k}.  If~$z^*$ is such that~$m_n = \Omega_n(z^*)$,
then the optimal~$\nu$ is supported at~$0$ and~$z^*$.  Solving the
resulting system yields the theorem.
\end{proof}

Table~\ref{tab:Rn-values} contains some values for~$\vartheta(\R^n, k)$.

\begin{table}[tb]
  \footnotesize
  \centerfloat
  \begin{tabular}{cccccccccc}
    $n$ / $k$&3&4&5&6&7&8&9&10&11\\
    2&0.64355&---&---&---&---&---&---&---&---\\
    3&0.42849&0.69138&---&---&---&---&---&---&---\\
    4&0.29346&0.49798&0.73225&---&---&---&---&---&---\\
    5&0.20374&0.36768&0.55035&0.76580&---&---&---&---&---\\
    6&0.15225&0.28471&0.42777&0.60262&0.79563&---&---&---&---\\
    7&0.11866&0.22740&0.34071&0.48493&0.64681&0.81972&---&---&---\\
    8&0.09339&0.18405&0.27471&0.39559&0.53374&0.68268&0.83882&---&---\\
    9&0.07387&0.15030&0.22864&0.33042&0.44903&0.57816&0.71431&0.85537&---\\
    10&0.05846&0.12340&0.19194&0.27851&0.38158&0.49496&0.61521&0.74026&0.86882
  \end{tabular}
  \bigskip

  \caption{The bound~$\vartheta(\R^n, k)$ for~$n = 2$, \dots,~$10$ and~$k =
    3$, \dots,~$11$, with values of~$n$ on each row and of~$k$ on each
    column.}
  \label{tab:Rn-values}
\end{table}


\section{Exponential density decay}
\label{sec:decay}

In this section we analyze the asymptotic behavior
of~$\vartheta(S^{n-1}, k, t)$ and~$\vartheta(\R^n, k)$ as functions of~$n$,
proving Theorem~\ref{thm:main-decay}.

The main step in our analysis is to understand the asymptotic behavior
of
\[
  M_n(t) = \min\{\, \unijac{n}{j}(t) : j \geq 0\,\},
\]
as defined in~\eqref{eq:mnt}.  For~$t \in [-1, 0)$ we
have~$M_n(t) \leq \unijac{n}{1}(t) = t$, and so~$M_n(t)$ does not
approach~$0$.  We have seen in~\S\ref{sec:sphere} that~$M_n(0) = -1/(n-1)$,
so for~$t=0$ we have that~$M_n(t)$ approaches~$0$ linearly fast as~$n$
grows.  Things get interesting when~$t \in (0, 1)$: then~$M_n(t)$
approaches~$0$ exponentially fast as~$n$ grows.

\begin{theorem}
  \label{thm:mnt-asymptotic}
  For every~$t \in (0, 1)$ there is~$c \in (0, 1)$ such
  that~$|M_n(t)| \leq (c +\nobreak o(1))^n$.
\end{theorem}

We will need the following lemma showing that, for every~$t \in (0, 1)$,
if~$j = \Omega(n)$, then~$|\unijac{n}{j}(t)|$ decays exponentially in~$n$.
Theorem~\ref{thm:mnt-asymptotic} will follow from an application of this
lemma after we show that the minimum in~\eqref{eq:mnt} is attained for
some~$j^* = \Omega(n)$.  The statement of the lemma is quite precise since
we later want to do a more detailed analysis of the base of the
exponential.  The proof is a refinement of the analysis carried out by
Schoenberg~\cite{Schoenberg1942}.

\begin{lemma}
  \label{lem:jac-decay}
  If for~$\theta \in (0, \pi)$ and~$\delta \in (0, \pi/2)$ we write
  \[
    C = (\cos^2\theta + \sin^2\theta \sin^2\delta)^{1/2},
  \]
  then~$|\unijac{n}{j}(\cos\theta)| \leq \pi n^{1/2} \cos^{n-3}\delta +
  C^j$ for all~$n \geq 3$.
\end{lemma}

\begin{proof}
An integral representation for the ultraspherical polynomials due to
Gegenbauer (take~$\lambda = (n - 2)/2$ in Theorem~6.7.4 from Andrews,
Askey, and Roy~\cite{AndrewsAR1999}) gives us the formula
\[
  \unijac{n}{j}(\cos\theta) = R(n)^{-1} \int_0^\pi F(\phi)^j
  \sin^{n-3}\phi\, d\phi,
\]
where
\[
  F(\phi) = \cos\theta + i\sin\theta\cos\phi\qquad\text{and}\qquad
  R(n) = \int_0^\pi \sin^{n-3} \phi\, d\phi.
\]

Note that~$|F(\phi)|^2 = \cos^2\theta + \sin^2\theta \cos^2\phi$ and
that~$|F(\phi)| \leq 1$.  Split the integration domain into the
intervals~$[0, \pi/2-\delta]$, $[\pi/2-\delta, \pi/2+\delta]$,
and~$[\pi/2+\delta, \pi]$ to obtain
\[
  \begin{split}
    |\unijac{n}{j}(\cos\theta)| &\leq R(n)^{-1} \int_0^\pi |F(\phi)|^j
    \sin^{n-3}\phi\, d\phi\\
    &\leq 2R(n)^{-1} \int_0^{\pi/2-\delta} \sin^{n-3} \phi\, d\phi +
    R(n)^{-1} \int_{\pi/2-\delta}^{\pi/2+\delta} |F(\phi)|^j
    \sin^{n-3}\phi\, d\phi.
  \end{split}
\]

For the first term above, note that
\[
  R(n) = \frac{\pi^{1/2} \Gamma(n/2 - 1)}{\Gamma((n-1)/2)}.
\]
Take~$x = (n-2)/2$ and~$a = 1/2$ in~(7) of Wendel~\cite{Wendel1948} to get
\[
  R(n)^{-1} \leq \pi^{-1/2} ((n-2)/2)^{1/2} < n^{1/2}.
\]
Now
\[
  \begin{split}
    2R(n)^{-1} \int_0^{\pi/2-\delta} \sin^{n-3} \phi\, d\phi &\leq
    2n^{1/2} \int_0^{\pi/2-\delta} \sin^{n-3}(\pi/2-\delta)\, d\phi\\
    &=2n^{1/2}(\pi/2-\delta)\cos^{n-3}\delta\\
    &\leq \pi n^{1/2} \cos^{n-3}\delta.
  \end{split}
\]

For the second term we get directly
\[
  R(n)^{-1} \int_{\pi/2-\delta}^{\pi/2+\delta} |F(\phi)|^j \sin^{n-3}\phi\,
  d\phi \leq R(n)^{-1} \int_{\pi/2-\delta}^{\pi/2+\delta} C^j
  \sin^{n-3}\phi\, d\phi \leq C^j,
\]
and we are done.
\end{proof}

We can now prove the theorem.

\begin{proof}[Proof of Theorem~\ref{thm:mnt-asymptotic}]
Our strategy is to find a lower bound on the largest~$j_0$ such that
$\unijac{n}{j}(t) \geq 0$ for all~$j \leq j_0$.  Then we know that~$M_n(t)$
is attained by some~$j \geq j_0$, and we can use Lemma~\ref{lem:jac-decay}
to estimate~$|M_n(t)|$.

Recall~\cite[Theorem~3.3.2]{Szego1975} that the zeros of~$\unijac{n}{j}$
are all in~$[-1, 1]$ and that the rightmost zero of~$\unijac{n}{j+1}$ is to
the right of the rightmost zero of~$\unijac{n}{j}$.  Let~$C^\lambda_j$
denote the ultraspherical (or Gegenbauer) polynomial with
parameter~$\lambda$ and degree~$j$, so
\begin{equation}
  \label{eq:jac-gegen}
  \unijac{n}{j}(t) = \frac{C^{(n-2)/2}_j(t)}{C^{(n-2)/2}_j(1)}.
\end{equation}

Let~$x_j$ be the largest zero of~$C^\lambda_j$.  Elbert and
Laforgia~\cite[p.~94]{ElbertL1990} show that, for~$\lambda \geq 0$,
\[
  x_j^2 < \frac{j^2 + 2\lambda j}{(j + \lambda)^2}.
\]
If for a given~$j$ we have that
\begin{equation}
  \label{eq:k-estimate}
  \frac{j^2 + 2\lambda j}{(j + \lambda)^2} \leq t^2,
\end{equation}
then we know that the rightmost zero of~$C^\lambda_j$ is to the left
of~$t$, and so~$C^\lambda_j(t) \geq 0$.

The left-hand side in~\eqref{eq:k-estimate} is increasing in~$j$; let us
estimate the largest~$j$ for which~\eqref{eq:k-estimate} holds.  We want
\[
  j^2 + 2\lambda j - t^2(j+\lambda)^2 \leq 0.
\]
The left-hand side above is quadratic in~$j$ and, since~$t^2 < 1$, the
coefficient of~$j^2$ is positive.  So all we have to do is to compute the
largest root of the left-hand side, which is~$2a(t)\lambda$,
where~$a(t) = ((1 -t^2)^{-1/2} - 1)/2$.

Hence for~$j \leq 2a(t)\lambda$ we have~$C^\lambda_j(t) \geq 0$.
From~\eqref{eq:jac-gegen} we see that~$\unijac{n}{j}(t) \geq 0$ if
\[
  j \leq a(t)n - 2a(t).
\]
Now plug the right-hand side above into the upper bound of
Lemma~\ref{lem:jac-decay} to get
\[
  \begin{split}
    |M_n(t)| &\leq (\pi n^{1/2} \cos^{-3} \delta) \cos^n\delta +
    C^{a(t)n - 2a(t)}\\
    &=O(n^{1/2})(\cos\delta)^n + O(1)(C^{a(t)})^n,
  \end{split}
\]
with~$C$ as defined in Lemma~\ref{lem:jac-decay} with $\cos\theta = t$.
For any choice of~$\delta \in (0, \pi/2)$, we have that~$\cos\delta$,
$C \in (0, 1)$ and, since~$a(t) > 0$ for all~$t \in (0, 1)$, the theorem
follows by taking any~$c$ such that~$\max\{\cos\delta, C^{a(t)}\} < c < 1$.
\end{proof}

We now get exponential decay for~$\vartheta(S^{n-1}, k, t)$ for
any~$k \geq 3$ and~$t \in (0, 1)$.  Indeed, consider the
recurrence~$F_0 = t$ and~$F_i = F_{i-1}/(1+F_{i-1})$ for~$i \geq 1$, whose
solution is~$F_i = t/(1+it)$.  Using Theorem~\ref{thm:mnt-asymptotic} to
develop our analytic solution~\eqref{eq:sphere-theta-expr}, we get
\begin{equation}
  \label{eq:sphere-asymp-general}
  \vartheta(S^{n-1}, k, t) \sim \sum_{i=0}^{k-2} |M_{n-i}(F_i)| =
  \sum_{i=0}^{k-2} |M_{n-i}(t / (1 + it))|,
\end{equation}
where $a_n \sim b_n$ means that $\lim_{n \rightarrow \infty} a_n/b_n = 1$.
Since~$t / (1 + it) > 0$ for all~$i$, each term decays exponentially fast,
and so we get exponential decay for the sum.

We also get exponential decay for~$\vartheta(\R^n, k)$ for any~$k \geq 3$.
Indeed, from Theorem~\ref{thm:theta-rn-k} we have that
\begin{equation}
  \label{eq:rn-asymp-general}
  \vartheta(\R^n, k) \sim |m_n| + \sum_{i=0}^{k-3} |M_{n-i}(1 / (2+i))|.
\end{equation}
From Theorem~\ref{thm:mnt-asymptotic} we know that every term in the
summation above decays exponentially fast.  Bachoc, Nebe, Oliveira, and
Vallentin~\cite{BachocNOV2009} give an asymptotic bound for~$|m_n|$ that
shows that it also decays exponentially in~$n$, namely
\[
  |m_n| \leq (2/e + o(1))^{n/2} = (0.8577\ldots + o(1))^n.
\]
This finishes the proof of Theorem~\ref{thm:main-decay}.


\subsection{Explicit bounds}
\label{sec:exp-bounds}

We now compute explicit constants $c(k, t)$ and $c(k)$ which can serve as
bases for the exponentials in Theorem~\ref{thm:main-decay}, in particular
obtaining the bounds advertised in~\S\ref{sec:Ramsey}.

The constant~$c$ given in Theorem~\ref{thm:mnt-asymptotic} depends on~$t$.
Following the proof, we can find the best constant for every~$t \in (0, 1)$
by finding~$\delta \in (0, \pi/2)$ such that $\cos\delta = C^{a(t)}$, that
is, by solving the equation
\begin{equation}
  \label{eq:best-constant}
  \cos^4\delta = (t^2 + (1 - t^2) \sin^2\delta)^{(1-t^2)^{-1/2} - 1}
\end{equation}
and taking~$c = \cos\delta > 0$.

For any given~$t \in (0, 1)$ it is easy to solve~\eqref{eq:best-constant}
numerically.  For~$t = 1/2$ we get~$\cos\delta = 0.95621\ldots$ as a
solution, and so~$|M_n(1/2)| \leq (0.95622 + o(1))^n$, leading to the the
bound
\[
  \vartheta(\R^n, 3) \sim |M_n(1/2)| \leq (0.95622 + o(1))^n.
\] 
Figure~\ref{fig:best-c} shows a plot of the best constant~$c$ for
every~$t \in (0, 1)$.

\begin{figure}
  \begin{center}
    \includegraphics{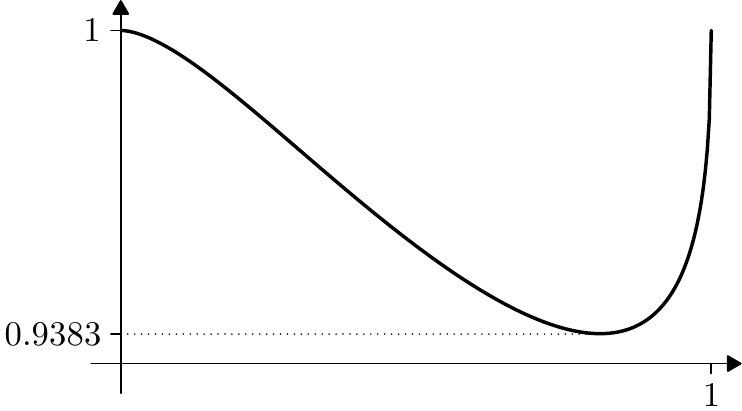}
  \end{center}

  \caption{The best constant~$c$ obtained in our proof of
    Theorem~\ref{thm:mnt-asymptotic} for each value of~$t \in (0, 1)$.}
  \label{fig:best-c}
\end{figure}

With a little extra work, it is possible to show that, for all~$k \geq 2$,
\begin{equation}
  \label{eq:rough}
  |M_n(1/k)| \leq \biggl(1-\frac{1}{9k^2}+o(1)\biggr)^n,
\end{equation}
whence
\[
  \vartheta(\R^n, k) \sim |M_{n-k+3}(1/(k-1))| \leq \biggl(1 -
  \frac{1}{9(k-1)^2} + o(1)\biggr)^n
\]
for all~$k \geq 3$.

Direct verification shows that~\eqref{eq:rough} holds for~$k = 2$, so let
us assume~$k \geq 3$.  Writing~$c$ for the (unique) positive solution
$\cos\delta$ of~\eqref{eq:best-constant} and taking~$\theta \in (0, \pi/2)$
such that $\cos\theta = t$, we can rewrite~\eqref{eq:best-constant} in the
more convenient form
\begin{equation}
  \label{eq:best-constant2}
  c^{4\sin\theta/(1-\sin\theta)} = 1 - c^2 \sin^2\theta.
\end{equation}

Now say~$c = 1-x$ and use Bernoulli's inequality~$(1 + z)^r \geq 1 + rz$ to
get
\begin{align*}
  &(1-x)^{4\sin\theta/(1-\sin\theta)} \geq 1 -
    \frac{4\sin\theta}{1-\sin\theta} x\quad\text{and}\\
  &1 - (1-x)^2 \sin^2\theta \leq 1 - (1-2x) \sin^2\theta.
\end{align*}
Equating the left-hand sides of both inequalities above and solving
for~$x$, we get
\[
  c = 1-x \leq 1 - \frac{\sin\theta (1-\sin\theta)}{4 + 2\sin\theta
    (1-\sin\theta)}.
\]
In particular, when $\cos\theta = 1/k$ we get
\begin{align*}
  |M_n(1/k)| &\leq \biggl(1 - \frac{1}{4k^2
               (1 + \sqrt{k^2/(k^2-1)}) + 2} + o(1)\biggr)^n \\
             &\leq \biggl(1 - \frac{1}{9k^2} + o(1)\biggr)^n
\end{align*}
for all $k \geq 3$.


\section*{Acknowledgments}

We would like to thank Christine Bachoc for helpful discussions and
comments at an early stage of this work.  We are also thankful to both
anonymous referees whose suggestions certainly improved this paper.


\end{document}